\documentclass[a4paper,11pt]{article}

\newif\ifcomments\commentsfalse

 \usepackage[colorinlistoftodos,bordercolor=orange,backgroundcolor=orange!20,linecolor=orange,textsize=tiny]{todonotes}
\usepackage{array}
\usepackage{amsthm}
\usepackage{amsmath,amscd,amssymb}
\usepackage{latexsym}
\usepackage{epsfig}
\usepackage{tikz}
\usetikzlibrary{intersections}
\usepackage{longtable}

\newtheorem{theorem} {Theorem}[section]
\newtheorem{proposition}[theorem]{Proposition}
\newtheorem{lemma}[theorem]{Lemma}
\newtheorem{corollary}[theorem]{Corollary}

\newtheorem{question}[theorem]{Question}

\newtheorem{conjecture}[theorem]{Conjecture}
\theoremstyle{definition}

\newtheorem{remark}[theorem]{Remark}

\newtheorem{definition}[theorem]{Definition}
\renewcommand{\AA}{\mathbb A}

\newcommand{\CC}{\mathbb C}

\newcommand{\GG}{\mathbb G}

\newcommand{\NN}{\mathbb N}

\newcommand{\PP}{\mathbb P}
\newcommand{\QQ}{\mathbb Q}
\newcommand{\RR}{\mathbb R}

\newcommand{\ZZ}{\mathbb Z}
\newcommand{\imic}{\cong}

\newcommand{\minus}{\smallsetminus}

\newcommand{\Sum}{\sum\limits}

\renewcommand{\Tilde}{\widetilde}

\newcommand{\Spec}{\mathop{\mathrm {Spec}}\nolimits}

\newcommand{\Ker}{\mathop{\mathrm {Ker}}\nolimits}

\newcommand{\codim}{\mathop{\mathrm {codim}}\nolimits}

\newcommand{\dist}{\mathop{\mathrm {dist}}\nolimits}

\newcommand{\vertices}{\mathop{\mathrm {Vx}}\nolimits}
\newcommand{\bda}{\mathbf a}
\newcommand{\bdb}{\mathbf b}

\newcommand{\bde}{\mathbf e}

\newcommand{\bdn}{\mathbf n}

\newcommand{\bdp}{\mathbf p}
\newcommand{\bdq}{\mathbf q}
\newcommand{\bdr}{\mathbf r}
\newcommand{\bds}{\mathbf s}

\newcommand{\bdv}{\mathbf v}
\newcommand{\bdw}{\mathbf w}
\newcommand{\bdx}{\mathbf x}
\newcommand{\bdy}{\mathbf y}
\newcommand{\bdz}{\mathbf z}

\newcommand\QQuint[5]{\ensuremath{\scriptstyle{#1,#2,#3,#4,#5}}}
\newcommand\Vn{\ensuremath{-1+\Ssum n_i}}
\newcommand\Deltae[1]{\ensuremath{\Delta_{#1,\eps}}}
\newcommand{\cGe}[1]{\ensuremath{\mathcal G_{#1,\eps}}}
\newcommand\Deltape{\Deltae{\bdp}}
\newcommand\Deltaoe[1]{\Delta^\circ_{#1,\eps}}
\newcommand\Ssum{{\Sigma}}

\newcommand{\Conv}{\mathop{\mathrm {Conv}}\nolimits}
\newcommand{\Vx}{\mathop{\mathrm {Vx}}\nolimits}
\newcommand{\RRp}{\RR_{\ge 0}}
\newcommand{\outside}{\Omega}
\newcommand{\Lambdap}{\Lambda_\bdp}
\newcommand{\eps}{\varepsilon}
\newcommand{\nullset}{\varnothing}
\begin{document}
\title{Blowups with log canonical singularities}
\author{G.K.~Sankaran\thanks{
Department of Mathematical Sciences, University of Bath, Bath
BA2~7AY, UK, {\tt G.K.Sankaran@bath.ac.uk}}
\and 
Francisco Santos
\thanks{
Departamento de Matem\'aticas, Estad\'istica y Computaci\'on,
Universidad de Can\-tabria,
39005 Santander, Spain,
{\tt francisco.santos@unican.es}.
Supported by grant MTM2017-83750-P of the Spanish Ministry of Economy
and Competitiveness and by  the Einstein
Foundation Berlin under grant EVF-2015-230 
}
}

\maketitle

\begin{abstract}
We show that the minimum weight of a weighted blow-up of $\AA^d$ with
$\eps$-log canonical singularities is bounded by a constant
depending only on $\eps$ and $d$. This was conjectured by Birkar. 

Using the recent classification of $4$-dimensional empty simplices 
by Iglesias-Vali\~no and Santos, we work out an explicit bound
for blowups of $\AA^4$ with terminal singularities: the smallest weight is 
always at most $32$, and
at most $6$ in all but finitely many cases.
\end{abstract}

%\tableofcontents

\section{Introduction}\label{sect:introduction}

At a meeting of the COW seminar at City, University of
London on 7th February 2018, Caucher Birkar asked the following
question.

\begin{question}\label{qu:simple}
Denote by $\AA^4_\bdn$ the weighted blowup of $\AA^4$ at $0\in\AA^4$
with coprime weights $\bdn=(n_1,n_2,n_3,n_4)\in\NN^4$. If $\AA^4_\bdn$
has terminal singularities, is the smallest of the weights bounded?
\end{question}

By ``coprime'' we mean only that $\bdn$ is primitive: we do not require the
weights to be pairwise coprime.

This is a simplified version of a more ambitious conjecture.

\begin{conjecture}[Birkar]\label{qu:advanced}
Denote by $\AA^d_\bdn$ the weighted blowup of $\AA^d$ at $0\in\AA^d$
with coprime weights $\bdn=(n_1,\ldots,n_d)\in\NN^d$. If $\AA^d_\bdn$ has
$\eps$-log canonical singularities, then the smallest of the
weights is bounded by a constant depending only on $d$ and $\eps$.
\end{conjecture}

Our main result, Theorem~\ref{thm:Birkar}, is a proof of
Conjecture~\ref{qu:advanced}.

\begin{theorem}\label{thm:Birkar}
In each fixed dimension $d$ and for each $\eps\in (0,1]$ there is
  an integer $\ell_{\eps,d}\in \NN$ such that if
  $\bdn=(n_1,\dots,n_d)\in\NN^d$ is primitive and the weighted blowup
  $\AA^d_\bdn$ has only $\eps$-log canonical singularities then
  $n_{\min}:= \min\{n_1,\dots,n_d\} \le \ell_{\eps,d}$. 
\end{theorem}

Our proof relies on a general result about subgroups of
$\RR^n$ that miss a given open set, due to Lawrence~\cite{La}, which
we state here as Theorem~\ref{thm:families}. The connection of that
result to terminal and canonical singularities, and to hollow and
empty simplices, was first noticed by A.~Borisov~\cite{Bor}.
Independently of us, and by somewhat different methods,
Y.~Chen~\cite{Chen} has proved Conjecture~\ref{qu:advanced} for the
case $d=3$.

We also give a precise answer to
Question~\ref{qu:simple}.

\begin{theorem}\label{thm:terminal4}
If the weighted blowup $\AA^4_\bdn$ has terminal singularities then
$n_{\min} \le 32$. Moreover, with finitely many exceptions
$n_{\min} \le 6$.
\end{theorem}

The proof of this statement relies on the complete classification of
empty simplices in dimension four due to Iglesias-Vali\~no and
Santos~\cite{IS2}.  The bound of $6$ is attained by the infinite
family of blowups with $\bdn=(6,10,15,n)$, which have terminal
singularities whenever $n$ is coprime with $30$ (see
Remark~\ref{rem:nmin6}).  The bound of $32$ is attained only by the
blowup with $\bdn=(32, 41,71, 102)$.  There are a total of 1784
blowups of $\AA^4$ with $n_{\min}>6$; the number of them for each
value of $n_{\min}$ is listed in Proposition~\ref{prop:k4}.

\medskip 

These results extend a theorem of
Kawakita~\cite[Theorem~3.5]{Kawakita}, which says that a weighted
blowup $\AA^3_\bdn$ is terminal if and only if the weights are
$(1,a,b)$ with $a$ and $b$ coprime. Kawakita's result also follows
from our methods: see Corollary~\ref{coro:kawakita} below.

The context of~\cite{Kawakita} is the Sarkisov program, in particular
birational rigidity. To investigate Sarkisov links involving a Fano
3-fold $F$ of Picard rank~$1$ requires in principle an understanding
of all possible divisorial contractions in the Mori program with
target~$F$. The main outcome of~\cite{Kawakita} is that any divisorial
contraction in the Mori program with centre a smooth point is a
weighted blowup, and~\cite[Theorem~3.5]{Kawakita} says that the
weights must then be $(1,a,b)$.

This is important because, at least in dimension~3, we understand
divisorial contractions well if we know their sources, but not so well
if we know their targets. So~\cite{Kawakita} provides a description
of all possible baskets of singularities in a terminal 3-fold with a
divisorial contraction whose centre is a smooth point. This may be
thought of as a relative boundedness result, showing that exceptional
divisors are weighted projective planes of the form $\PP(1,a,b)$.

Birkar's Conjecture~\ref{qu:advanced} arises analogously in his
work~\cite{Birkar} on boundedness of log Calabi-Yau fibrations. One
way to view it is as a local version of the BAB conjecture, in a quite
special case.

\medskip
\noindent
\emph{Acknowledgements:} Some background on birational geometry was
supplied to GKS by Anne-Sophie Kaloghiros.  The explanations here
relating these results to their wider context are largely hers, but
errors and omissions in such explanations are definitely ours.  Parts
of this work were carried out while GKS was visiting Fukuoka
University and KIAS, Seoul: he thanks both for hospitality and a
helpful environment.  We also thank the organisers of MEGA 2019
(Madrid), where the two authors first met and discussed
these questions.

\section{Singularities and simplices}\label{sect:singlecone}

Geometrically, our approach is to use toric geometry to rephrase the
problem in terms of polytopes. We shall be working in $\RR^d$ with its
standard basis $\bde_1=(1,0,\ldots,0),\ldots,\bde_d$. We shall
frequently need to add up the coordinates of a vector, so we write
$\Ssum x_i$ to abbreviate $\Sum\nolimits_{i=1}^d x_i$.

\begin{definition}\label{def:polytopes}
Let $\Lambda\subseteq \RR^d$ be a lattice: that is, a finitely
generated free abelian subgroup of rank~$d$ such that
$\RR^d=\Lambda\otimes\RR$. A \emph{polytope} $\Pi$ in $\RR^d$ is a
bounded intersection of finitely many closed half-spaces. A point
$\bdv\in \Pi$ is a \emph{vertex} if $\Pi\cap H=\{\bdv\}$ for some
affine hyperplane $H\subset \RR^d$: we denote the set of vertices of
$\Pi$ by $\vertices(\Pi)$. The convex hull of a set $X\subset \RR^d$
is denoted $\Conv(X)$: a polytope $\Pi$ is always equal to the convex
hull $\Conv(\Vx(\Pi))$ of its vertices. $\Pi$ is a \emph{lattice
  polytope} if $\vertices(\Pi)\subset \Lambda$.
\end{definition}

The next definition is usually made only for the case where
$\Gamma$ is a lattice and $\Pi$ is a lattice polytope,
but we need it in a more general setting.
\begin{definition}\label{def:empty}
Fix a subgroup $\Gamma$ of $\RR^d$. We say that a polytope $\Pi$ is
\emph{hollow with respect to $\Gamma$} if $\Pi\cap \Gamma\subseteq
\partial\Pi$ and \emph{empty with respect to $\Gamma$} if $\Pi\cap
\Gamma\subseteq \vertices(\Pi)$.  We omit ``with respect to $\Gamma$''
when $\Gamma$ is understood.
\end{definition} 

Let $\sigma=\sum \RRp\bdw_r$ be a nondegenerate closed rational
polyhedral cone in $\RR^d$, where $\bdw_r\in \Lambda$ are primitive
generators of the rays of $\sigma$. We denote by $\Delta(\sigma)$ the
lattice polytope $\Conv(\{0\}\cup\{\bdw_i\})$, and let $X_\sigma$ be
the affine variety $\Spec\CC[\sigma^\vee\cap \Lambda]$, as usual in
toric geometry. With this notation, $X_\sigma$ is $\QQ$-Gorenstein if
and only if all the $\bdw_i$ lie in an affine hyperplane, and is
$\QQ$-factorial if and only if $\sigma$ is simplicial; that is, if
$\Delta(\sigma)$ is a simplex.

The following fundamental fact is well known.
\begin{lemma}\label{lem:terminal}
Let $\eps\in(0,1]$. Then:
  \begin{enumerate}
  \item[(a)] $X_\sigma$ is $\eps$-log terminal if and only if
    $\eps \Delta(\sigma)$ is an empty polytope.
  \item[(b)] $X_\sigma$ is $\eps$-log canonical if and only if
    $\eps \Delta(\sigma)$ is hollow and all nonzero lattice points in it lie
    in facets not containing the origin.   
   \end{enumerate}
\end{lemma}
\begin{proof}
$X_\sigma$ is $\eps$-log canonical if and only if for some (hence any)
  birational morphism $f\colon Y\to X_\sigma$ with $Y$ smooth, the
  discrepancies $e_j$ defined by $K_Y-f^*K_X=\sum_je_jE_j$ (with $E_j$
  being $f$-exceptional prime divisors) satisfy $e_j\ge -1+\eps$. To
  check this, consider a toric resolution $f\colon Y=Y_\Sigma\to
  X_\sigma$ obtained by subdividing $\sigma$ into a regular fan
  $\Sigma$. The exceptional divisors are given by some rays $\rho_j$
  spanned by primitive $\bdr_j\in\Lambda$. The $\QQ$-divisors $K_Y$
  and $f^*K_{X_\sigma}$ are given by support functions $h_Y$ and
  $h_{X_\sigma}$ as in \cite[Proposition~2.1(v)]{Od}. The function
  $h_Y$ satisfies $h_Y(\bdr_j)=h_Y(\bdw_i)=1$, while $h_{X_\sigma}$ is
  linear and is determined by $h_{X_\sigma}(\bdw_i)=0$. Therefore
  $e_j=-1+h_{X_\sigma}(\bdr_j)$, so in part~(b) we have
  $h_{X_\sigma}(\bdr)\ge \eps$, for all $\bdr\in\Lambda$. The result
  follows at once from this: part~(a) is identical, replacing $e_j\ge
  -1+\eps$ by $e_j> -1+\eps$.
\end{proof}

In particular, since canonical is the same as $1$-log canonical,
$X_\sigma$ has $\QQ$-factorial canonical singularities if and only if
$\Delta(\sigma)$ is a hollow simplex with $\Delta(\sigma)\cap \Lambda
\minus \{0\}$ contained in the facet opposite to the origin.

Any nonnegative primitive integer vector
$\bdn=(n_1,\ldots,n_d)\in\NN^d$ induces a weighted blowup
$\AA^d_\bdn$, which is the toric variety associated with the fan in
$\RR^d$ (and the lattice $\ZZ^d$) that consists of all the faces of
the cones $\sigma_\bdn^j=\RRp\bdn+\sum_{i\neq j}\RRp\bde_i$. Note that
all such faces are contained in $\RRp^d$, and that the $\sigma_\bdn^j$
are simplicial so $\AA^d_\bdn$ always has $\QQ$-factorial
singularities.

The standard simplex in $\RR^d$ is
$\Delta:=\Delta(\RRp^d)=\Conv(\{0,\bde_1,\ldots,\bde_d\})$ and its
interior is denoted $\Delta^\circ$. That is,
\[
\Delta^\circ=\{\bdx\in\RR^d\mid \Ssum x_i < 1 \text{ and } \forall i \, x_i>0\}.
\]
The facet of $\Delta$ opposite to the origin, which is
$\Conv(\{\bde_1,\ldots,\bde_d\})$, is denoted by $\Delta_1$.

For any non-zero $\bdn\in \NN^d$ we set
$\Delta_\bdn=\Conv(\{\bde_1,\ldots,\bde_d,\bdn\})$.

\begin{proposition}\label{prop:onecone}
For $\eps \in (0,1]$
  \begin{enumerate}
  \item[(a)] $\AA^d_\bdn$ has $\eps$-log terminal singularities if
    and only if $\eps \Delta_\bdn$ is empty.
  \item[(b)] $\AA^d_\bdn$ has $\eps$-log canonical singularities
    if and only if $\eps \Delta_\bdn$ is hollow.
  \end{enumerate}  
\end{proposition}

\begin{proof}
  \begin{itemize}
    \item[(a)] The singularities of $\AA^d_\bdn$ are $\eps$-log
      terminal if and only if all the polytopes $\eps
      \Delta(\sigma_\bdn^j)$ are empty: that is, if $\bigcup_{j=1}^d
      \eps \Delta(\sigma_\bdn^j)$ is empty.  But
 \begin{eqnarray*}
   \bigcup_{i=1}^n \eps\Delta_{\sigma_\bdn^i}&=&\eps\Conv(0,\bde_1,\ldots,\bde_d,\bdn)\\
     &=& \eps\Conv(0,\bde_1,\ldots,\bde_d)\cup
   \eps\Conv(\bde_1,\ldots,\bde_d,\bdn)\\
   &=& \eps\Delta \cup \eps\Delta_\bdn
 \end{eqnarray*}
  and $\eps\Conv(\{0,\bde_1,\ldots,\bde_d\})$ is empty anyway.
\item[(b)] All lattice points of
  $\bigcup_{i=1}^n\eps\Delta(\sigma_\bdn^i)$ other than the origin lie
  in $\eps\Delta_\bdn$ by construction. Hence they all lie in facets
  not containing the origin if and only if they do not lie in the
  interior of $\eps\Delta_\bdn$ or in $\eps\Delta_\bdn\cap \eps\Delta
  = \eps\Conv(\{\bde_1,\ldots,\bde_d\})=\eps\Delta_1$. The latter is
  empty, and except for the trivial case $\eps=1$ has no lattice
  points among its vertices either.
\end{itemize}
\end{proof}

The following change of coordinates sends the simplex $\Delta_\bdn$ of
Proposition~\ref{prop:onecone} to the standard simplex $\Delta$, which
will be useful for us.

\begin{lemma}\label{lem:change}
Let $\bdn=(n_1,\dots,n_d) \in \RRp^d$ be a non-negative vector with
$\Ssum n_i >1$.  Then the unique affine-linear transformation sending
$\bdn$ to the origin and fixing all of $\bde_1,\dots,\bde_d$ sends the
origin to $\bdn/(\Vn)$.
\end{lemma}

\begin{proof}
The unique (modulo multiplication by a scalar) affine dependences
among $\{0, \bde_1,\dots,\bde_d, \bdn\}$ and among $\{\bdn/(\Vn),
\bde_1,\dots,\bde_d, 0\}$ are the same one: its coefficients are
$(1-\Ssum n_i ,n_1,\dots,n_d,-1)$.
\end{proof}

\begin{corollary}\label{coro:onecone}
Let $\bdn\in \NN^d$. Define $V=\Vn$ and $\bdp = \frac{1}{V}\bdn\in
\QQ^d$.  Let $\Lambdap=\ZZ^d+\ZZ\bdp$ be the lattice generated by
$\bdp$ and $\ZZ^d$.  Then, for any $\eps\in (0,1]$:
  \begin{enumerate}
  \item[(a)] $\AA^d_\bdn$ has $\eps$-log terminal singularities if
    and only if $\Deltape=\bdp + \eps(\Delta-\bdp)$ is empty with
    respect to the lattice $\Lambdap$.
  \item[(b)] $\AA^d_\bdn$ has $\eps$-log canonical singularities
    if and only if $\Deltape$ is hollow with respect to the lattice
    $\Lambdap$.
  \end{enumerate}  
\end{corollary}

\begin{proof}
This is just Proposition~\ref{prop:onecone}, rephrased via the change
of coordinates of Lemma~\ref{lem:change}. The notation here will be
used more widely: see Definition~\ref{def:phollow} below.
\end{proof}

\section{$\eps$-log canonical singularities}
\label{sec:mainproof}

This section is devoted to the proof of Theorem~\ref{thm:Birkar}.

\subsection{Lawrence's Theorem and hollow points}\label{subsect:Lawrence}

Apart from the relation between $\eps$-log canonical singularities
and hollow simplices described in Corollary~\ref{coro:onecone}, our
main technical tool is the following result of Jim Lawrence (see also
\cite{Bor}).

\begin{theorem}[Lawrence~\protect{\cite[Theorem~1]{La}}]
  \label{thm:families}
Fix $d\in \NN$ and an open subset $U\subset \RR^d$, and let $\GG$
be a closed subgroup of $\RR^d$ containing $\ZZ^d$. Then there are
only finitely many maximal subgroups $G<\GG$ such that
$\ZZ^d\subset G$ and $G\cap U=\nullset$.
\end{theorem}

In other words, any subgroup of $\GG$ that contains $\ZZ^d$ and misses
$U$ is contained in one (at least) of finitely many such subgroups of
$\GG$.

These maximal subgroups $G$ are automatically closed. Hence $G$ is a
Lie subgroup of $\RR^d$, and its identity component, which we call
$L$, is a linear subspace of dimension equal to $\dim G$. Some of the
groups containing $\ZZ^d$ that we consider below are not closed,
however.

The relation to our problem comes from the fact that the lattice
$\Lambdap$ in Corollary~\ref{coro:onecone} is a subgroup of $\RR^d$
containing $\ZZ^d$. This implies, for example, that taking
$U=\Delta^\circ$, we may interpret the case $\eps=1$ of
Corollary~\ref{coro:onecone}(b) as saying that if $\AA^d_\bdn$ has
only canonical singularities then $\bdp$ lies in one of finitely many
subgroups of $\RR^d$ containing $\ZZ^d$ and not intersecting
$\Delta^\circ$.

Our aim is to extend this approach to any value of $\eps\in(0,1]$. We
first extend the notation introduced in Corollary~\ref{coro:onecone},
using Definition~\ref{def:empty}.

\begin{definition}\label{def:phollow}
We define 
\[
\outside := \RRp^d\minus \Delta = \{\bdx\in \RR^d \mid \Ssum x_i >1
\text{ and } \forall i\, x_i\ge 0\}.
\]
For each point $\bdp\in \outside$:
\begin{enumerate}
\item[(a)] We call the number $V:= \frac1{-1+\Ssum p_i}\in \RRp$ the
  \emph{index} of $\bdp$. The entries of the vector $\bdn:=V\bdp\in
  \RRp^d$ are called the \emph{weights} of $\bdp$, and the smallest of
  them is called the \emph{smallest weight} $n_{\min}=n_{\min}(\bdp)$
  of $\bdp$. 
\item[(b)] We put $\Deltape= \bdp + \eps(\Delta-\bdp)$ and
  $\Lambdap=\ZZ^d+\ZZ\bdp$.
\item[(c)] We say that $\bdp$ \emph{is $\eps$-hollow} if $\Deltape$ is
  hollow with respect to the group $\Lambdap$.
\end{enumerate}
\end{definition}

The notation in Definition~\ref{def:phollow}(a) is compatible with the notation of
Corollary~\ref{coro:onecone} because
\[
-1+\Ssum n_i=-1+V\Ssum p_i=-1+V\left(\frac{1}{V}+1\right)=V,
\]
but at this stage we do not require the weights to be integers: $V$
and $\bdn$ need not even be rational, so the group $\Lambdap$ may not
be a lattice.

Observe that $\Deltape$ is $\Delta$ shrunk towards $\bdp$ by a factor $\eps$,
so it is a simplex with facets parallel to the facets of $\Delta$.

\subsection{The canonical case of Birkar's conjecture}
\label{subsect:canonicalbirkar}

We let $H_0=\{\bdx\mid \Ssum x_i=0\}$ and $H_1=\{\bdx\mid \Ssum
x_i=1\}$. Thus $H_1$ is the affine hyperplane containing $\Delta_1$
and $H_0$ is the linear hyperplane parallel to it.  Let
$\Delta_1^\circ$ denote the relative interior of $\Delta_1$.

Fix a linear subspace $L\subset \RR^d$, of codimension~$k$. Assuming
that $L\not\subseteq H_0$ we are going to prove a bound $\ell_L$,
depending only on $L$, for the minimum weight of every point $\bdp\in
\outside$ such that $L+\bdp$ does not meet $\Delta_1^\circ$.

For this, let $\pi_L\colon \RR^d \to \RR^d/L\cong \RR^k$ be the
canonical projection along $L$, let $\bds_i = \pi_L(\bde_i)$, and let
$S=\{0,\bds_1,\dots, \bds_d\}$, so that $\Conv(S) =
\pi_L(\Delta)$. The condition $L\not\subseteq H_1$ implies that no
affine hyperplane in $\RR^d/L$, in particular no facet of $\Conv(S)$,
contains $\{\bds_1,\dots,\bds_d\}$. This makes the minimum in the
following statement well-defined.

\begin{proposition}\label{prop:facetbound}
Suppose that $L\subseteq \RR^d$ is a linear subspace not contained in
$H_1$. For each facet-supporting hyperplane $H$ of $\pi_L(\Delta)$ let
\[
\ell_H:= \min_{\bds_i\not\in H} \frac{\dist(H,0)}{\dist(H,\bds_i)},
\]
and let $\ell_L=\max_H \ell_H$.  Then every point $\bdp\in \outside$
such that $\bdp+L$ does not meet $\Delta_1^\circ$ has $n_{\min}(\bdp)\le
\ell_L$.
\end{proposition}

\begin{remark}
Let $k=d-\dim L$.  In $\RR^d/L \cong \RR^k$, an affine hyperplane $H$
is expressed as $H=\{\bdx\in \RR^k\mid f(\bdx)=c\}$, where $f\colon
\RR^k\to \RR$ is a linear functional.  For $\bdy\in\RR^k$, we define
the distance $\dist(H,\bdy)=|f(\bdy)-c|$. This depends on the choice
of $f$, which is only unique up to a scalar and, implicitly, on the
choice of isomorphism $\RR^d/L \cong \RR^k$.  But in the statement of
Proposition~\ref{prop:facetbound} and the rest of this section we only
consider \emph{ratios} of two distances, which do not depend on
choice. In Section~\ref{sect:lowdim} we shall need to be more
definite.
\end{remark}

\begin{proof}
Since $(\bdp+L)\cap\Delta_1^\circ=\nullset$ and $\bdp\in \outside$,
we also have $(\bdp+L)\cap\Delta^\circ=\nullset$, and the point
$\pi_L(\bdp)$ is not in the interior of $\Conv(S)$. Hence there is a
facet-supporting hyperplane $H$ of $\Conv(S)$ that weakly separates
$\pi_L(\bdp)$ from $\Conv(S)$. Let $\Tilde H= \pi_L^{-1}(H)$, which is
a hyperplane weakly separating $L+\bdp$ from $\Delta$ (but is not
necessarily facet-supporting for $\Delta$).

If $0\in \Tilde H$ then, in order for $\bdp$ to be in $\outside$, one
of the coordinates of $\bdp$, hence one of the weights of $\bdp$, must
be zero. Thus we assume $0\not\in \Tilde H$ and we can find an
$\bda\in\RR^d$ such that $\Tilde{H}=\{\bdx\in \RR^d \mid
\bda.\bdx=1\}$, where $\bda.\bdx:=\sum_{i=1}^d a_i x_i$ is the usual
Euclidean inner product.
  
Since $\Tilde{H}$ weakly separates $\Delta$ from $\bdp$ we have
$\sum_i a_ip_i=\bda.\bdp\ge 1$ but $\bda.\bdx\le 1$ for every $\bdx\in
\Delta$; in particular, $a_i=\bda.\bde_i\le 1$ for every~$i$. Thus
\[
\sum_{i=1}^d(1-a_i)n_i = \sum_{i=1}^d n_i - V\sum_{i=1}^d a_i p_i \le
(V+1) - V = 1.
\]
Since the terms in the first sum are non-negative, $(1-a_i)n_i\le 1$
for every $i$.

Observe that $\dist(\Tilde{H},0)=1$ and
$\dist(\Tilde{H},\bde_i)=(1-\bda.\bde_i)$ so
\[
\frac{\dist(H,\bds_i)}{\dist(H, 0)}=
\frac{\dist (\Tilde{H},\bde_i)}{\dist(\Tilde{H}, 0)}=
1-a_i.
\]
Hence, for any $i$ with $\bds_i\not\in H$ (which exists, because
otherwise we would have $\Tilde H=\{\Ssum x_i =1\}=H_1$ and that would
imply $L\subset H_0$) we have
\[
n_i\le \frac1{1-a_i} =\frac{\dist (H,0)}{\dist(H, \bds_i)}.
\]
Thus $n_{\min}(\bdp)\le \ell_H$. This does not yet give a bound for
$n_{\min}(\bdp)$ because $H$ depends on $\bdp$, but $H$ is one of the
finitely many facet-supporting hyperplanes of $\pi_L(\Delta)$, so
$n_{\min}(\bdp)\le \max_H \ell_H = \ell_L$ as claimed.
\end{proof}

Although we give below a separate proof of the general case, it is
interesting to observe that Proposition~\ref{prop:facetbound} leads to
the following easy proof of the canonical case of
Theorem~\ref{thm:Birkar}.

\begin{proof}[Proof of Theorem~\ref{thm:Birkar} for $\eps=1$]
By Theorem~\ref{thm:families} there is a finite collection
$\{G_1,\ldots,G_t\}$ of closed subgroups of $\RR^d$ containing
$\ZZ^d$ and not meeting $\Delta^\circ$, such that any subgroup of
$\RR^d$ containing $\ZZ^d$ and not meeting $\Delta^\circ$ is contained
in one of them. We denote $L_j$ the identity component of $G_j$.

If $L_j\subset H_0$, then the quotient $G_j/(G_j\cap H_0) \cong
\pi_{H_0} (G_j)$ is a discrete subgroup of $\RR^d/H_0 \cong \RR$. Let
$y$ be the minimum of $\pi_{H_0} (G_j)$ in the interval $(1,\infty)$
and define $\ell_{G_j}=1/(-1+y)$.  Then the index (and hence each
weight) of every $\bdp \in G_j\cap \outside$ is bounded by
$\ell_{G_j}$.

If $L_j\not \subset H_0$, then Proposition~\ref{prop:facetbound}
applies, since $L_j+\bdp\subset G_j$ does not meet $\Delta^\circ$.
The proposition gives us an $\ell_{G_j}=\ell_{L_j}$ (depending only on
$L_j$) with $n_{\min}(\bdp)\le \ell_{G_j}$ for every $\bdp \in G_j\cap
\outside$.

We can then take $\ell_{1,d}=\max_{j=1,\dots, t}\ell_{G_j}$. Indeed,
let $\bdn \in \NN^d$ be such that $\AA^d_\bdn$ has only canonical
singularities. As above, let $V=\Vn$ and let $\bdp=\frac1V\bdn$, which
lies in $\outside$.  By Corollary~\ref{coro:onecone} the lattice
$\Lambda_\bdp= \ZZ^d+\ZZ\bdp$ does not meet $\Delta^\circ$ and is thus
contained in some $G_j$ from our list.  Thus, $n_{\min} =
n_{\min}(\bdp) \le \ell_{G_j} \le \ell_{1,d}$.
\end{proof}

\subsection{Local weight bound}\label{subsect:localwt}

In this section we examine the situation near a given point $\bdx$ of
$\Delta_1$ and show the following.

\begin{proposition}\label{prop:local_boundary}
Let $\eps\in (0,1]$ and $d\in \NN$ be fixed. Then, for each point
$\bdx \in \Delta_1$, there is a non-negative integer $\ell_\bdx \in
\NN$ and an open neighbourhood $W_\bdx$ of $\bdx$ in $\RR^d$, such
that if $\bdp\in \outside\cap W_\bdx$ is $\eps$-hollow then its
smallest weight $n_{\min}(\bdp)$ satisfies $n_{\min}(\bdp) \le \ell_\bdx$.
\end{proposition}

To prove this we introduce the following notation. For each set $U$
with $\bdx\in U\subseteq \RR^d$ we define $\Deltae{U}=\bigcap_{\bdq\in
  U}\Deltae{\bdq}$, and we let $\cGe{U}$ be the family of all
subgroups of $\RR^d$ containing $\ZZ^d$ and not meeting
$\Deltae{U}^\circ$. Observe that
\begin{equation*}
U\supseteq U'
\qquad \Rightarrow \qquad
\Deltae{U}\subseteq\Deltae{U'}
\qquad \Rightarrow \qquad
\cGe{U} \supseteq \cGe{U'}.
\end{equation*}

We are interested in the case where $U$ is a neighbourhood of
$\bdx$.

\begin{lemma}\label{lem:deltainteriors}
Let $B_1 \supset B_2 \supset \dots$ be a countable base of
neighbourhoods of $\bdx$, so that $\bigcap_{r\in \NN} B_r = \{\bdx\}$.
Then $\bigcup_{r\in\NN}\Deltaoe{B_r}=\Deltaoe{\bdx}$.
\end{lemma}

\begin{proof} The inclusion
  $\bigcup_{r\in\NN}\Deltaoe{B_r}\subseteq\Deltaoe{\bdx}$ is
  immediate. For the other direction, if $\bdy\in\Deltaoe{\bdx}$ then
  \begin{eqnarray*}
\bdx\in \{\bdz\mid \bdy\in \Deltaoe{\bdz}\}&=&\{\bdz\mid \exists
\bdw\in \eps\Delta^\circ \text{ such that }
\bdy=\bdz(1-\eps)+\bdw\}\\
&=&\{\bdz\mid \bdy-\bdz(1-\eps)\in \eps\Delta^\circ\},
  \end{eqnarray*}
which is open because $\eps\Delta^\circ$ is open and $\bdz\mapsto
\bdy-\bdz(1-\eps)$ is continuous.

Hence $\bdy\in\Deltaoe{\bdz}$ for all $\bdz$ in some neighbourhood of
$\bdx$, and in particular for all $\bdz\in B_r$ for some sufficiently
large~$r$. Hence $\bdy\in \bigcup_{r\in\NN}\Deltaoe{B_r}$.
\end{proof}

By analogy with Definition~\ref{def:phollow} we say that a closed
group $G$ with identity component $L$ is \emph{$\eps$-hollow at
  $\bdx$} if $G \cap (\bdx+L)\cap\Deltaoe{\bdx} = \nullset$.

Observe that this includes all closed groups with $\bdx\not\in G$,
since in this case $G \cap (\bdx+L)$ is already empty. Our next two
lemmas prepare the proof of Proposition~\ref{prop:local_boundary},
dealing separately with groups that are and are not $\eps$-hollow at
$\bdx$.

\begin{lemma}\label{lem:Lfiniteness}
Every $\bdx\in \Delta_1$  has an open neighbourhood $U_\bdx$
such that every closed group in $\cGe{U_\bdx}$ is $\eps$-hollow at $\bdx$.
\end{lemma}

\begin{proof}
Let $B_1 \supset B_2 \supset \dots$ be a countable base of
neighbourhoods of $\bdx$. We will prove the following, which has
Lemma~\ref{lem:Lfiniteness} as the case $k=0$:

\begin{quote}
\emph{For every $k\in\{0,\dots,d\}$ there is an $r$ such that every
  closed group of dimension $\ge k$ in $\cGe{B_r}$ is $\eps$-hollow at
  $\bdx$}.
\end{quote}

The proof of this is by induction on $d-k$.  The base case $k=d$ is
trivial since the only group of dimension $d$ is the whole space
$\RR^d$, and this group does not lie in $\cGe{B_1}$. (We assume that
$\Deltae{B_1}$ has non-empty interior: Lemma~\ref{lem:deltainteriors}
allows us to do this.)

Now, for a fixed $k$, our induction hypothesis is that there is an $r$
such that every closed group of dimension greater than $k$ in $\cGe{B_r}$ is
$\eps$-hollow at $\bdx$. That is, every closed group in $\cGe{B_{r}}$
that is \emph{not} $\eps$-hollow at $\bdx$ has dimension at most $k$.
By Theorem~\ref{thm:families}, $\cGe{B_{r}}$ contains finitely many
maximal groups, all closed. Let us denote $G_1,\dots G_t$ the ones of
dimension $k$ that are not $\eps$-hollow (if any), and let
$L_1,\dots,L_t$ be their corresponding identity components. Observe
that, although $\cGe{B_{r}}$ may contain additional non-$\eps$-hollow
groups of dimension $k$, apart from the $G_i$'s, any such group must
be contained in one of the $G_i$'s and, in particular, its identity
component must equal the corresponding $L_i$.

For each $i\in \{1,\dots,t\}$, since $G_i$ is non-$\eps$-hollow,
$\bdx+L_i$ meets $\Deltaoe{\bdx}$; by Lemma~\ref{lem:deltainteriors},
$\bdx+L_i$ meets $\Deltaoe{B_{r_i}}$ for some $r_i$. In particular,
$\cGe{B_{r_i}}$ contains neither $G_i$ nor any other group whose
identity component equals $L_i$. Obviously, the same holds for any
$r\ge r_i$.

Hence, taking $r' = \max \{r_1,\dots, r_t\}$ we have that
$\cGe{B_{r'}}$ does not contain any group with identity component
equal to any of the $L_i$'s. Since $B_{r'}\supset B_r$ we have
$\cGe{B_{r'}} \subset \cGe{B_{r}}$, and hence all the
non-$\eps$-hollow groups in $\cGe{B_{r'}}$ are non-$\eps$-hollow
groups in $\cGe{B_{r}}$ too, but necessarily of smaller dimension.
\end{proof}

\begin{lemma}\label{lemma:bound_for_G}
Let $\bdx \in \Delta_1$ and let $G$ be a closed group containing
$\ZZ^d$ and $\eps$-hollow at $\bdx$.  Then there is a neighbourhood
$W_G$ of $\bdx$ and a natural number $\ell_G$ such that every $\bdp
\in \Omega\cap G \cap W_G$ has $n_{\min}(\bdp) \le \ell_G$.
\end{lemma}

\begin{proof}
Let $L$ be the identity component of $G$. There are three
possibilities:
\begin{itemize}
\item If $\bdx \not\in G$, simply take $W_G=\RR^d \setminus G$ and
  $\ell_G=0$.

\item If $L\subset H_0$, then $\pi_{H_0}(G)=G/(G\cap H_0) \subset \RR$ is
  discrete. Let $s$ be its minimum in $(1,\infty)$. We can take $W_G=
  \{\bdp \mid \Ssum p_i < s\}$ and $\ell_G=0$, since $\Omega\cap G
  \cap W_G = \nullset$.
  
\item If $\bdx\in G$ and $L\not\subset H_0$, then $\bdx + L \subset G$
  but $(\bdx+L)\cap\Deltaoe{\bdx} = \nullset$, because $G$ is
  $\eps$-hollow. But then $L+\bdx$ does not meet $\Delta_1^\circ$, so
  we may apply Proposition~\ref{prop:facetbound} to $L$. We then
  get an $\ell_G$ such that for every $\bdp \in \Omega\cap( \bdx+L)$
  we have that the minimum weight of $\bdp$ is bounded by $\ell_L$. We
  can then take $W_G = \RR^d\setminus (G \setminus (\bdx +L))$, so
  that $G \cap W_G = \bdx + L$ and $\Omega\cap G \cap W_G =
  \Omega\cap( \bdx+L)$.  \qedhere
\end{itemize}
\end{proof}

We can now prove Proposition~\ref{prop:local_boundary}.

\begin{proof}[Proof of Proposition~\ref{prop:local_boundary}]
By Lemma~\ref{lem:Lfiniteness}, $\bdx$ has an open neighbourhood
$U_\bdx$ such that every group in $\cGe{U_\bdx}$ that contains $\bdx$
is $\eps$-hollow.  By Theorem~\ref{thm:families}, $\cGe{U_\bdx}$ has a
finite number of maximal elements, all closed and $\eps$-hollow at
$\bdx$, which we denote $G_1,\dots,G_t$. By
Lemma~\ref{lemma:bound_for_G}, each $G_i$ gives a neighbourhood $W_i$
of $\bdx$ and a natural number $\ell_i$ such that every $\bdp \in
\Omega\cap G_i \cap W_i$ has $n_{\min}(\bdp) \le \ell_i$.

Now it is enough to take $W_\bdx= U_\bdx \cap (\bigcap_i W_i)$ and
$\ell_\bdx = \max \ell_i$.  Indeed, let $\bdp\in W_\bdx\cap\outside$
be $\eps$-hollow, so that $\Deltae{\bdp}\cap \Lambda_\bdp =
\nullset$. Since $\bdp\in W_\bdx$, we have $\Deltae{\bdp}\supset
\Deltae{W_\bdx}\supset \Deltae{U_\bdx}$. In particular, the group
$\Lambda_\bdp$ is in $\cGe{U_\bdx}$, and hence is contained in one of
the $G_i$'s.  Thus $\bdp \in \Omega\cap G_i \cap W_i$.
\end{proof}

\subsection{The general case of Birkar's conjecture}\label{subsect:generalbirkar}

We are now in a position to give the proof of
Theorem~\ref{thm:Birkar}, settling Conjecture~\ref{qu:advanced}
completely.

\begin{proof}[Proof of Theorem~\ref{thm:Birkar}] Fix
$\eps\in(0,1]$. For each $\bdx\in \Delta_1$, choose $\ell_\bdx$ and
$W_\bdx$ as in Proposition~\ref{prop:local_boundary}, with
$\ell_\bdx$ as small as possible. For a non-negative integer
$\ell$, define $\Delta_1(\ell) := \{ \bdx \in \Delta_1 \mid
\ell_\bdx \le \ell\}$. Then $\Delta_1(\ell)$ is relatively open in
$\Delta_1$, because if $\bdy\in W_\bdx\cap \Delta_1$ then
$\ell_\bdy<\ell_\bdx$.  Moreover, the $(\Delta_1(\ell))_{\ell\in \NN}$ obviously
form an increasing sequence and they cover $\Delta_1$. 
Observe, for example, that
$\Delta_1^\circ\subseteq \Delta_1(0)$, because if
$\bdx\in\Delta_1^\circ$ and $G\cap(\bdx+L)$ meets $\Delta_1^\circ$
then $L\subset H_0$. Put differently, Proposition~\ref{prop:facetbound} is not
needed on $\Delta_1^\circ$.

By compactness, there is an open subset
$W=\bigcup_{\bdx\in\Delta_1^\circ}W_\bdx$ and an integer $\ell_W$ such
that $\Delta_1\subset W$ and every $\eps$-hollow $\bdp\in \outside\cap
W$ has $n_{\min}(\bdp)\le\ell_W$.
On the other hand, if $\bdp\in 2\Omega$ then $V<1$, and since
$\outside \setminus(2\outside\cup W)$ is compact, the index
(hence the minimum weight) of all $\bdp \in \outside \setminus U$ has
a global upper bound.
\end{proof}

\section{Terminal and canonical bounds}\label{sect:lowdim}

Throughout this section we take $\eps=1$, so that we are considering
only canonical and terminal singularities. In these cases 
we compute more explicit bounds, assuming that $\dim L$ or $\codim L$
is small. Combining these bounds with the classification of empty
$4$-simplices in~\cite{IS2} we give precise bounds in the terminal
$4$-fold case: that is, a precise answer to Question~\ref{qu:simple}. 

\subsection{Bounds in terms of width}\label{subsect:facetwidths}

We first rework the bound of Proposition~\ref{prop:facetbound} in
terms of the lattice width of $\Conv(S)=\pi_L(\Delta)$.

\begin{definition}\label{def:facetwidth}
A linear functional $f\colon\RR^d\to\RR$ is called \emph{primitive}
with respect to a lattice $\Lambda$ if $f(\Lambda)=\ZZ$.

The \emph{width} of a lattice polytope $\Pi$ in the direction of $f$ 
is the length of the interval $f(\Pi)$. Its \emph{facet width} 
with respect to a facet $F$ is the width in the direction of the
unique (up to a sign) primitive linear functional that is constant on $F$. 
\end{definition}

Let $G \subseteq \RR^d$ be a closed group containing $\ZZ^d$ and not
meeting $\Delta^\circ$, with identity component $L$. We keep the
notation from Subsection~\ref{subsect:canonicalbirkar}, and we let
$\Lambda_G=\pi_L(G)$, which is a lattice in $\RR^d/L$, and put
\[
\ell_G=\max\{n_{\min}(\bdp)\mid \bdp\in \outside \cap G\},
\]
i.e.\ the best possible bound for the smallest weight in $G$.

\begin{proposition}\label{prop:facetwidth}
$\ell_G$ is bounded by the maximum facet width of $\pi_L(\Delta)$ with
  respect to $\Lambda_G$.
\end{proposition}

\begin{proof}
Suppose first that $L\not\subset H_0$ and let $H$ be a
facet-supporting hyperplane of $\pi_L(\Delta)=\Conv(S)$.  We normalise
the distance to $H$ by taking $f$ to be the primitive linear
functional constant on $H$ and $\dist(H,\bdx)=|f(\bdx)-f(H)|$. Then
$1\le \dist(H,\bds_i)\in\NN$ for every $\bds_i\not\in H$ and
$\dist(H,0)$ is bounded above by the facet width with respect to the
facet contained in $H$. Hence the statement follows from
Proposition~\ref{prop:facetbound}.

If $L\subset H_0$ then $\pi_L(H_1)$ is a facet-supporting hyperplane
of $\pi_L(\Delta)$. If $\bdp \in \outside \cap G$ then
$\pi_L(\bdp)\in\Lambda_G$ and is strictly separated from
$\pi_L(\Delta)$ by $\pi_L(H_1)$. So if $f$ is the primitive linear
functional constant on $\pi_L(H_1)$, then $f_1:=f(\pi_L(H_1))$ is the
facet width of $\pi_L(\Delta)$ with respect to $\pi_L(H_1)$, and
$f(\bdp)\ge f_1+1$. Hence $\Ssum p_i\ge \frac{f_1+1}{f_1}$, so $V\le
f_1$ and therefore $n_{\min}(\bdp)\le f_1$.
\end{proof}

\begin{corollary}\label{coro:cases}
With the notation of Proposition~\ref{prop:facetwidth},
\begin{enumerate}
\item[(a)] If $\pi_L(\Delta)$ has width equal to $1$ in  \emph{some} 
  lattice direction then $\ell_G\in\{0,1\}$. 
  This is always the case if $\dim L=d-1$. 
\item[(b)] If $\dim L=d-2$, then $\ell_G\in \{0,1,2\}$.
\end{enumerate}
\end{corollary}

\begin{proof}
\begin{enumerate}
\item[(a)] Let $f$ be a primitive functional giving width~$1$ to
  $\Delta/L$, and $\Tilde f$ its pull-back to $\RR^d$. Then $G':=
  G+\Ker(\Tilde f)$ is a closed group containing $G$ and not
  intersecting $\Delta^\circ$, which implies $\ell_G \le \ell_{G'}$.
  
  Thus there is no loss of generality in assuming $\dim L=d-1$.  In
  this case $L=\Ker(\Tilde f)$, so $\pi_{L}(\Delta)=f(\Delta)$ is a
  hollow lattice polytope of dimension $1$, that is, a unit
  segment. This has facet width $1$ with respect to every facet, so
  Proposition~\ref{prop:facetwidth} gives the statement.
  
\item[(b)] Here $\pi_L(\Delta)$ is a hollow lattice polytope of
  dimension~$2$. This implies $\pi_L(\Delta)$ either has width $1$ or
  equals (modulo an affine isomorphism of the lattice) the triangle
  $\Conv((0,0),(2,0),(0,2))$ (see, e.g., \cite{Hurkens}). This
  triangle has width $2$ with respect its to all its three facets.
  \qedhere
\end{enumerate}
\end{proof}

We can now recover Kawakita's result on the terminal weighted blowups
in dimension~$3$.

\begin{corollary}[\protect{\cite[Theorem~3.5]{Kawakita}}]
\label{coro:kawakita}
The weighted blowup $\AA^3_\bdn$ has terminal singularities if and
only if the weights are $(1,a,b)$, with $a$ and $b$ coprime.
\end{corollary}
\begin{proof}
This follows immediately from Corollary~\ref{coro:cases}(a) and the
theorem of White~\cite{White} that all empty $3$-simplices have
width~$1$.
\end{proof}

\subsection{Groups of dimension~$1$}\label{subsect:dim1groups}

For our application to $d=4$ in Subsection~\ref{subsect:dim4} below,
we want to consider the case $\dim L=1$ more carefully. In this case
let $(a_1,\dots, a_d)\in\ZZ^d$ be a primitive integer vector in $L$,
which is unique up to sign, and let $a_0:=\sum_{i=1}^d a_i$.  The
vector $\bda:=(a_0,\dots,a_d)\in \ZZ^{d+1}$ is called the
\emph{$(d+1)$-tuple} of $L$.  We assume $L\not\subseteq H_0$, which is
equivalent to $a_0\ne 0$.

\begin{lemma}\label{lem:dim1general}
Suppose $\bdp\in \outside$ and that $\dim L=1$, and that
$(\bdp+L)\cap\Delta^\circ=\nullset$. Then $n_{\min}(\bdp)\le
\max_{i=1,\dots,d}\{-a_i/a_0\}$.
\end{lemma}

\begin{proof}
The set $S=\{0,\bds_1,\dots,\bds_d\}$ affinely spans $\RR^d/L\imic
\RR^{d-1}$ and has $d+1$ points, so it has a unique (modulo a scalar
factor) affine dependence. Since $\sum_{i=1}^d a_i \bde_i\in L$, the
coefficient vector of that dependence is precisely $\bda$.

To bound the minimum weight we use
Proposition~\ref{prop:facetbound}. Let $H$ be a facet-supporting
hyperplane of $\Conv(S)$.  If $0\in H$ then $\ell_H=0$ in
Proposition~\ref{prop:facetbound}. If $0\not\in H$ then, since
$L\not\subset H_0$, there must be an $i$ with $\bds_i\not \in H$. Thus
$H$ contains all of $S$ except for $0$ and a single $\bds_i$. Applying
the affine dependence $\bda$ to the affine functional vanishing on $H$
gives $\dist(H,0) \, a_0 + \dist(H,\bds_i)\, a_i =0 $,
which finishes the proof since
\[
\min_{\bds_j\not\in H} \frac{\dist(H,0) }{ \dist(H,\bds_j)} = 
\frac{\dist(H,0) }{ \dist(H,\bds_i)} = 
- \frac{a_i }{ a_0} .
\qedhere
\]
\end{proof}

We also have the following alternative bound, which is better than the
previous one in a few critical cases.

\begin{lemma}\label{lem:dim1special}
Let $\bdp\in\outside$ be such that $\bdn=V\bdp \in \NN^d$, where
$V=\frac{1}{-1+\Ssum p_i}$ as usual. Suppose that there is a proper
subset $J\subset\{1,\dots,d\}$ such that
\[
\sum_{i\in J} p_i - s \sum_{i=1}^d p_i \in \ZZ
\]
for a positive integer $s$. Then either
$\sum_{i\in J}n_j\le s$ or else $n_i=0$ for all $i\not\in J$.
\end{lemma}

\begin{proof}
Multiplying the equation in the statement by $V$ we obtain that
\[
\sum_{i\in J}n_i - s (V+1) \in V\ZZ,
\]
so $\sum_{i\in J}n_i \equiv s \pmod V$. Since $\Ssum n_i = V+1$,
either $n_i = 0$ for every $i\not\in J$, or $\sum_{i\in J} n_i \le V$. The
latter, together with $\sum_{i\in J}n_i \equiv s \pmod V$, implies
$\sum_{i\in J} n_i \le s$. 
\end{proof}

\subsection{Terminal $4$-fold case}\label{subsect:dim4}

Now we consider the case $d=4$, where there is an extensive
history. Notice that another interpretation of
Corollary~\ref{coro:onecone} is that $\AA^d_\bdn$ has terminal (or
canonical) singularities if and only if the cyclic quotient
singularity $\frac{1}{V}\bdn$ is terminal (or canonical), where
$V=\Vn$.

In fact any non-Gorenstein terminal quotient singularity in
dimension~$4$ is cyclic, but this fails in higher dimension:
see~\cite{Barile} for both of these facts. The singularity
$\frac{1}{V}\bdn$ is never Gorenstein, but we note for
completeness that Gorenstein cyclic terminal $4$-fold singularities
were classified in~\cite{MS}, and Gorenstein non-cyclic terminal
$4$-fold singularities in~\cite{Ann}.

In dimension~$4$, a classification of non-Gorenstein terminal quotient
singularities was begun experimentally in~\cite{MMM}. The first
definite result was proved in~\cite{Sa} (another proof of the same
result may be found in~\cite{Bob}): together with the results
of~\cite{Bor} and~\cite{Barile}, it implies that the list
in~\cite{MMM} of such singularities of prime index is complete with
possibly finitely many exceptions. Note, however, that the claim made
in~\cite{Barile} that the results of~\cite{Sa} and~\cite{Bob} are
valid for composite index is incorrect, as was pointed out
in~\cite{BHHS}.

The complete classification of non-Gorenstein terminal quotient
singularities in dimension~$4$ was recently given in~\cite{IS2}, and
we use it to prove Theorem~\ref{thm:terminal4}.

In~\cite[Section 2]{IS2} hollow simplices are divided into \emph{fine
  families}. Two hollow lattice simplices $\Delta_1$ and $\Delta_2$ in
$\RR^d$, with $\Vx(\Delta_i)=\{\bdv_{ij}\}\subset\ZZ^d$, lie in the
same fine family if there is an integer $k\le d$ and integer affine
maps $\pi_i\colon\ZZ^d\to \ZZ^k$ such that
$\pi_1(\Vx(\Delta_1))=\pi_2(\Vx(\Delta_2))=S$ and $\Conv(S)$ is
hollow. Here $S=\{\bds_0, \dots, \bds_d\}$ is to be thought of as a
multiset: that is, there is a permutation $\sigma$ of $\{0,\ldots,d\}$
such that $\pi_1(\bdv_{1\sigma(j)})=\pi_2(\bdv_{2j})$ for all~$j$.

As before, if $G$ is a closed group containing $\ZZ^d$ and with $G\cap
\Delta^\circ = \nullset$ then
$\pi_L(\Delta)$ is a hollow lattice polytope with respect to the
lattice $\Lambda_G=\pi_L(G)$. Thus the rational points in $G$ parametrise (perhaps part
of) a fine family of hollow simplices: each
point $\bdp \in G\cap \QQ^d$ corresponds, as in
Corollary~\ref{coro:onecone}, to the standard simplex
$\Delta\subset\RR^d$ considered with respect to
$\Lambdap$. In this situation we say $\bdp$ is a
\emph{generating point} of that hollow simplex. This relation makes
Theorem~\ref{thm:families} equivalent to \cite[Corollary 2.7]{IS2}.

The case $L=\{0\}$ corresponds to the \emph{sporadic hollow simplices}
that do not project to hollow polytopes of lower dimension: more
generally, the codimension of $L$, which we have called $k$ here, is
the same as the parameter $k$ in \cite[Theorem 1.6]{IS2}.  In
particular, cases $k=1,2,3,4$ of~\cite[Theorem 1.6]{IS2} correspond
exactly to the cases $\dim L=3,2,1,0$ in our setting. We prove
Theorem~\ref{thm:terminal4} separately for each value of $k$. We have
already done $k=1$ and $k=2$.

\begin{proposition}\label{prop:k12}
If a blowup $\AA^4_\bdn$ of $\AA^4$ belongs to the case $k=1$  
 then $n_{\min}\le 1$, and if $k=2$ then $n_{\min}\le 2$.
\end{proposition}
\begin{proof}
  These are just parts (a) and (b) of Corollary~\ref{coro:cases}.
  \end{proof}

For the case $k=3$, the most interesting one, 
we analyse the bounds from
Subsection~\ref{subsect:dim1groups}. The \emph{index} of a
family parametrised by a group $G$ as above is defined to be the index
$|G:L+\ZZ^d|$. A family is called \emph{primitive} if its index is
$1$, and \emph{non-primitive} otherwise.
 
The classification in~\cite{IS2} for $k=3$ consists of two lists: one
of 29 primitive quintuples Q1--Q29 (the same as the list of quintuples
that appears in~\cite{MMM}), and one of 17 non-primitive quintuples
N1--N17.

A primitive family is fully determined by $L$. In the case $\dim L=1$
and $d=4$ we specify $L$ via a quintuple $\bdq=(q_1,\dots,q_5)$ with
$\sum q_i=0$, defined by the property that $\RR\bdq$ parametrises
$(L+\ZZ^4)/\ZZ^4$ in barycentric coordinates with respect to the
standard simplex.  As shown in \cite{IS2}, the quintuple $\bdq$ can
also be interpreted as the affine dependence among the points in
$S=\pi_L(\{0,\bde_1,\dots,\bde_4\})$. Thus, modulo a permutation of
the entries, $\bdq$ is the same as the vector $\bda=(a_0,\dots,a_4)$
that we used in Lemma~\ref{lem:dim1general}.  However, in order to
apply Lemma~\ref{lem:dim1general} we need to specify which of the
entries $q_l$ will be considered the distinguished entry $a_0$.
  
A more concrete interpretation of the quintuple is as follows: for
each $V\in \NN$, the family corresponding to $\bdq$ contains a unique
(modulo affine-integer isomorphism) hollow simplex of index $V$; the
generating point $\bdp$ of this simplex can be chosen to be
$\bdp=\frac1V(a_1,\dots,a_d)$, where $(a_1,\dots,a_d)$ is obtained
from $\bdq$ by deleting the entry $q_l=a_0$ corresponding to the origin and
permuting the rest.  The generating point is only important modulo
$\ZZ^4$.

In the non-primitive case a family is determined by not only $L$ or
$\bdq$, but also by information on the group $G/(L+\ZZ^4)$.
In~\cite{IS2} and in the table below this is expressed by adding to
$\bdq$ a vector of the form $V\bdr$ (or of the form $\pm V\bdr$, for
the non-primitive quintuples of index greater than~$2$, which are
N7--N17).  Observe, however, that the statement of
Lemma~\ref{lem:dim1general} depends only on $L$, so only the $\bdq$
part plays any role in it. The part $V\bdr$ is only relevant when we
apply Lemma~\ref{lem:dim1special}.  Since we will do this only for one
non-primitive case, namely N5, we defer the details on how to
interpret $V\bdr$ to when we need it.

We now list the quintuples, with the conventional labels Q1--Q29 and
N1--N17.

\begin{center}
\begin{tabular}{l | c |l | c ||l|c}
\hskip-2pt  Case &  Quintuple & 
\hskip-2pt  Case &  Quintuple & 
\hskip-2pt  Case &  Quintuple  \\[6pt]
\hline
&&&&&\\
  Q1 & \QQuint{9}{1}{-2}{-3}{-5}   &  Q18 & \QQuint{15}{1}{-3}{-5}{-8} & N1 & \QQuint{6+\frac{V}{2}}{1}{-2}{-2+\frac{V}{2}}{-3}\\

  Q2 & \QQuint{9}{2}{-1}{-4}{-6}   &  Q19 & \QQuint{15}{2}{-1}{-6}{-10} & N2 & \QQuint{4}{3}{-1}{-2+\frac{V}{2}}{-4+\frac{V}{2}}\\

  Q3 & \QQuint{12}{3}{-4}{-5}{-6}  &  Q20 & \QQuint{15}{4}{-2}{-5}{-12} & N3 & \QQuint{8}{1}{-2+\frac{V}{2}}{-3}{-4+\frac{V}{2}}\\
 
  Q4 & \QQuint{12}{2}{-3}{-4}{-7}  &   Q21 & \QQuint{18}{1}{-4}{-6}{-9} & N4 & \QQuint{6+\frac{V}{2}}{3}{-1}{-2+\frac{V}{2}}{-6}\\

  Q5 & \QQuint{9}{4}{-2}{-3}{-8}   &   Q22 & \QQuint{18}{2}{-5}{-6}{-9} & N5 & \QQuint{8}{3}{-1}{-4+\frac{V}{2}}{-6+\frac{V}{2}}\\

  Q6 & \QQuint{12}{1}{-2}{-3}{-8}  &   Q23 & \QQuint{18}{4}{-1}{-9}{-12} & N6 & \QQuint{12}{1}{-3}{-4+\frac{V}{2}}{-6+\frac{V}{2}}\\

  Q7 & \QQuint{12}{3}{-1}{-6}{-8}  &   Q24 & \QQuint{20}{1}{-4}{-7}{-10} & N7 & \QQuint{3}{1}{-1\pm\frac{V}{3}}{-1\pm\frac{2V}{3}}{-2}\\

  Q8 & \QQuint{15}{4}{-5}{-6}{-8}  &   Q25 & \QQuint{20}{1}{-3}{-8}{-10} & N8 & \QQuint{3}{2}{-1}{-1\pm\frac{2V}{3}}{-3\pm\frac{V}{3}}\\

  Q9 & \QQuint{12}{2}{-1}{-4}{-9}  &   Q26 & \QQuint{20}{3}{-4}{-9}{-10} & N9 & \QQuint{3}{2}{-1}{-2\pm\frac{V}{3}}{-2\pm\frac{2V}{3}}\\

  Q10 & \QQuint{10}{6}{-2}{-5}{-9} &  Q27 & \QQuint{20}{3}{-1}{-10}{-12} & N10 & \QQuint{4\pm\frac{V}{3}}{2}{-1}{-1\pm\frac{2V}{3}}{-4}\\

  Q11 & \QQuint{15}{1}{-2}{-5}{-9} &  Q28 &  \QQuint{24}{1}{-5}{-8}{-12} &  N11 & \QQuint{6}{1}{-2}{-2\pm\frac{2V}{3}}{-3\pm\frac{V}{3}}\\

  Q12 & \QQuint{12}{5}{-3}{-4}{-10} &  Q29 & \QQuint{30}{1}{-6}{-10}{-15} & N12 & \QQuint{6}{1}{-1\pm\frac{2V}{3}}{-2}{-4\pm\frac{V}{3}}\\

  Q13 & \QQuint{15}{2}{-3}{-4}{-10} & && N13 & \QQuint{4}{3}{-1\pm\frac{2V}{3}}{-2}{-4\pm\frac{V}{3}}\\

  Q14 & \QQuint{12}{1}{-3}{-4}{-6} & && N14 &  \QQuint{6}{3\pm\frac{V}{3}}{-1}{-2\pm\frac{V}{3}}{-6\pm\frac{V}{3}}\\

  Q15 & \QQuint{14}{1}{-3}{-5}{-7} & && N15 &  \QQuint{3\pm\frac{V}{4}}{2}{-1}{-1\pm\frac{V}{4}}{-3\pm\frac{V}{2}}\\

  Q16 & \QQuint{14}{3}{-1}{-7}{-9} & && N16 &  \QQuint{6}{1\pm\frac{V}{4}}{-1}{-3\pm\frac{V}{4}}{-3\pm\frac{V}{2}}\\

  Q17 & \QQuint{15}{7}{-3}{-5}{-14} & && N17 &  \QQuint{3}{1\pm\frac{V}{6}}{-1}{-1\pm\frac{V}{6}}{-2\pm\frac{2V}{3}}\\
\end{tabular}
\end{center}

In every case the entries are arranged so that 
\[
q_1>q_2> \ 0 \ >q_3\ge q_4 \ge q_5.
\]
With this convention, we have $\max\{-a_j/a_0\} \le -q_1/q_3$ if
$a_0\in \{q_1,q_2 \}$ and $\max\{-a_j/a_0\} \le -q_5/q_2$ if $a_0\in
\{ q_3,q_4,q_5\}$.  Thus Lemma~\ref{lem:dim1general} implies the
following.  Observe that in the hypotheses of this statement we can
write $<7$ instead of $\le 6$ since all weights are integers.

\begin{lemma}
\label{lem:mostquintuples}
If a quintuple $\bdq$ (primitive or not) written as above satisfies
\[
\max\{-q_1/q_3, -q_5/q_2\} < 7
\] 
then every blowup coming from that quintuple has $n_{\max}\le 6$.
\qed
\end{lemma}

With this, we are now ready to prove the main result in this section,
which gives Theorem~\ref{thm:terminal4} for the families with $\dim
L=1$, that is, $k=3$.

\begin{proposition}\label{prop:k3}
If a blowup $\AA^4_\bdn$ of $\AA^4$ belongs to the case $k=3$
(equivalently, $\dim L=1$) then $n_{\min}\le 6$.
\end{proposition}

\begin{proof}
The reader may easily check that the only cases where
Lemma~\ref{lem:mostquintuples} is not sufficient to prove a bound of 6
are the ones shown (with the ratio $q_1:-q_3$ or $-q_5:q_2$ that we do
get) in the table below.  In all the other cases, including the ones
marked \lq\lq---\rq\rq\ in the table, the ratios $q_1:-q_3$ and
$-q_5:q_2$ are strictly less than~$7$.  In the non-primitive
quintuples this check is especially easy, since none of them has
$-q_5>6$ and the only ones with $q_1>6$ are N3, N5, and N6.
\begin{center}
\begin{tabular}{c|c|c}
\text{quintuple} & $q_1:-q_3$ & $-q_5:q_2$ \\
\hline
Q2&9\ :\ 1&---\\
Q6&---&8\ :\ 1\\
Q7&12\ :\ 1&---\\
Q9&12\ :\ 1&---\\
Q11&15\ :\ 2&9\ :\ 1\\
Q15&---&7\ :\ 1\\
Q16&14\ :\ 1&---\\
Q18&---&8\ :\ 1\\
Q19&15\ :\ 1&---\\
\end{tabular}
\qquad
\begin{tabular}{c|c|c}
\text{quintuple} & $q_1:-q_3$ & $-q_5:q_2$ \\
\hline
Q20&15\ :\ 2&---\\
Q21&---&9\ :\ 1\\
Q23&18\ :\ 1&---\\
Q24&---&10\ :\ 1\\
Q25&---&10\ :\ 1\\
Q27&20\ :\ 1&---\\
Q28&---&12\ :\ 1\\
Q29&---&15\ :\ 1\\
N5&8\ :\ 1&---\\
\end{tabular}
\end{center}

Even where the bound exceeds $7$, the ratios $-q_5/q_1$ and $-q_1/q_4$
(hence also $-q_1/q_5$) are less than~$7$, which implies that for the
cases with $l=1,4,5$ the bound of Lemma~\ref{lem:dim1general} is at
most $6$ in every quintuple. Thus the eighteen quintuples in the table
correspond to nineteen pairs (quintuple, $l$) that need to be checked: one
of $l=2$ or $l=3$ for each of the quintuples, except for the quintuple
Q11 where we have to check both.

Sixteen of the nineteen cases are primitive quintuples in which
$q_2=1$ (if $l=2$) or $q_3=-1$ (if $l=3$). This is fortunate since in
these cases it is particularly simple to apply
Lemma~\ref{lem:dim1special}. Indeed:
\begin{itemize}
\item If $a_0=q_2=1$ then we can use $s=-q_3$ in the lemma, by letting
  $J$ be just one coordinate, the one corresponding to $q_3$.
\item If $a_0=q_3=-1$ then we can use $s=q_2$ in the lemma, by letting
  $J$ be just one coordinate, the one corresponding to $q_2$.
\end{itemize}
That is, in these sixteen cases we can use $-q_3$ and $q_2$ as bounds
instead of the bigger $-q_5$ and $q_1$, respectively. The worst value
obtained is~$6$, for Q29 with $l=2$.

For the last three remaining cases we also apply
Lemma~\ref{lem:dim1special} as follows:
\begin{itemize}
\item For Q11$=(15,1,-2,-5,-9)$ with $a_0=q_3=-2$, our generating
  point is $\bdp=\frac1V(15,1,-5,-9)$. Taking $J$ to be the first and
  fourth coordinates and $s=3$ we have 
  $\sum_{i\in J} p_i - s \sum_{i=1}^4 p_i = \frac1V((15-9) -3\cdot 2 )=0$.
  Thus, Lemma~\ref{lem:dim1special} gives $n_1+n_4\le 3$.

\item For Q20$=(15,4,-2,-5,-12)$ with $a_0=q_3=-2$, our generating
  point is $\bdp=\frac1V(15,4,-5,-12)$.  Taking $J$ to be the first
  and third coordinates and $s=5$ we have 
  $\sum_{i\in J} p_i - s \sum_{i=1}^d p_i = \frac1V((15-5)-5\cdot2) =0$.
  Thus, Lemma~\ref{lem:dim1special} gives $n_1+n_3\le 5$.

\item 
  For N5 the quintuple is expressed as
  $({8},{3},{-1},{-4+\frac{V}{2}},{-6+\frac{V}{2}})$, that is, as
  $\bdq+V\bdr$ with $\bdq=(8,3,-1,-4,-6)$ and $\bdr=\frac{1}2(0, 0, 0,
  1, 1 )$.  The interpretation of this is that hollow simplices in
  this family are those with generating point (in barycentric
  coordinates) equal to
\[
\frac1V({8},{3},{-1},{-4},{-6}) + \frac{1}2(0,  0,  0,  1,   1 ).
\]
See \cite{IS2} for more details.

Since  $l=3$,  we have to omit the third coordinate and get
\[
\bdp=\frac1V\left({8},{3},{-4+\frac{V}{2}},{-6+\frac{V}{2}}\right),
\]
whose sum of coordinates is equal to $1+\frac1V$.

Taking $J$ to be just the second coordinate and $s=3$ we have
\[
\sum_{i\in J} p_i - s \sum_{i=1}^d p_i = \frac3V-3\left(1+\frac1V\right) =-3\in \ZZ,
\]  
so Lemma~\ref{lem:dim1special} gives $n_2\le 2$.
\end{itemize}

Thus, in all cases we get a bound of at most $6$ for the smallest weight.
\end{proof}

\begin{remark}
\label{rem:nmin6}
The bounds obtained by these methods are not sharp for each individual
quintuple and choice of $l$, but the overall bound in
Proposition~\ref{prop:k3} is sharp.  For example, the blowup
$\AA^4_{(V-30,6,10,15)}$, arising from Q29 with $l=2$, has terminal
singularities whenever $V$ is coprime with $30$, and has minimum
weight equal to $6$ for every $V\ge 37$.  This gives an infinite
family of blowups of $\AA^4$ with terminal singularities and
$n_{\min}=6$.
\end{remark}

To finish the proof of Theorem~\ref{thm:terminal4} we need to look at
the case $k=4$, that is, at the 2641 sporadic terminal 4-simplices
enumerated in~\cite{IS2}. The full list is publicly
available, and each simplex is expressed as a pair $(V, \bdb)$ with
$V\in \NN$ and $\bdb \in (\ZZ_V)^5$ where, as before, $V$ equals the
(normalised) volume and $\frac1V\bdb$ are the barycentric coordinates
(modulo an integer vector, which does not affect the lattice) for a
generator of $\Lambda/ \ZZ^d$.

Each such simplex corresponds to five terminal quotient singularities
(perhaps not distinct, if the simplex has symmetries) but not all such
singularities correspond to blowups of $\AA^4$. The conditions for
that are that:
\begin{itemize}
\item the corresponding entry $b_l$ of $\bdb$ is coprime to $V$, so
  that by multiplying by a unit in $\ZZ_V$ we can assume that entry to
  be $-1$, and
\item after this multiplication, the representatives in
  $\{0,\dots,V-1\}$ of the other four entries (remember that they are
  only important modulo $V$) add up to $V+1$.
\end{itemize}
When these conditions hold, the other four entries are the weights of
a blowup of $\AA^4$.

We have computationally checked the $2641\times 5$ possibilities,
obtaining the results summarised in the following statement.

\begin{proposition}\label{prop:k4}
Among the $2641\times 5$ sporadic terminal quotient singularities of
dimension $4$ there are $4620$ blowups, all with $n_{\min}\le
32$. The number $B$ of sporadic blowups with each possible value of
$n_{\min}$ is as follows.
\end{proposition}
\begin{center}
\begin{tabular}{c|c}
$n_{\min}$ & $B$ \\
\hline
1&0\\
2&964\\ 
3&804\\ 
4&413\\ 
5&468\\ 
6&187\\ 
7&408\\ 
8&212\\ 
\end{tabular}
\qquad
\begin{tabular}{c|c}
$n_{\min}$ & $B$ \\
\hline
9&194\\ 
10&130\\ 
11&178\\ 
12&81\\ 
13&137\\ 
14&63\\ 
15&63\\ 
16&48\\ 
\end{tabular}
\qquad
\begin{tabular}{c|c}
$n_{\min}$ & $B$ \\
\hline
17&65\\ 
18&34\\ 
19&57\\ 
20&26\\ 
21&16\\ 
22&11\\ 
23&23\\ 
24&7\\ 
\end{tabular}
\qquad
\begin{tabular}{c|c}
$n_{\min}$ & $B$ \\
\hline
25&12\\ 
26&5\\ 
27&5\\ 
28&2\\ 
29&3\\ 
30&1\\ 
31&2\\ 
32&1\\
\end{tabular}
\end{center}
The unique blowup with $n_{\min}=32$ has $V=245$ and $\bdn=(32, 41,
71, 102)$.  The unique sporadic simplex of maximum volume $V=419$
produces two blowups with terminal singularities, with weight vectors
\[
(20, 57, 133, 210) 
\qquad\text{and}\qquad
(21, 60, 140, 199).
\]

Theorem~\ref{thm:terminal4} now simply summarises
Propositions~\ref{prop:k12}, \ref{prop:k3} and \ref{prop:k4}.

\bibliographystyle{alpha}

\end{document}